\documentclass[11pt,reqno]{amsart}
\usepackage{amsmath,amssymb}
\usepackage[T1]{fontenc}
\usepackage{mathrsfs}
\usepackage{graphics}
\usepackage{tikz}
\usetikzlibrary{patterns}
\usepackage{color}
\usepackage{caption}
\usepackage{subfigure}
\usepackage{latexsym}
\usepackage{graphicx}
\usepackage{float}
\usepackage{extarrows}
\usepackage{epstopdf}
\usepackage{appendix}

\newtheorem{satz}{Proposition}[section]
\newtheorem{lem}[satz]{Lemma} 
\newtheorem{remark}[satz]{Remark}
\newtheorem{thm}[satz]{Theorem}

\newtheorem{cor}[satz]{Corollary}
\newtheorem{prop}[satz]{Proposition} 

\definecolor{gray}{gray}{0.50}
\definecolor{lred}{rgb}{1.0,0.5,0.5}
\definecolor{dgreen}{rgb}{0,1,1}

\numberwithin{equation}{section}

\newcommand{\chookrightarrow}{\mathrel{\lhook\joinrel\relbar\kern-.8ex\joinrel\lhook\joinrel\rightarrow}}

\newcommand{\R}{\mathbb{R}}

\newcommand{\N}{\mathbb{N}}		   
\newcommand{\Z}{\mathbb{Z}}

\newcommand{\dd}{\mathrm{d}}
\newcommand{\e}{\varepsilon}

\normalsize
\normalsize
\DeclareMathOperator{\Ima}{Im}
\DeclareMathOperator{\sech}{sech}

\definecolor{luh-dark-blue}{rgb}{0.0, 0.313, 0.608}
\definecolor{lred}{rgb}{1.0,0.5,0.5}

\usepackage[colorlinks=true]{hyperref}
\hypersetup{urlcolor=luh-dark-blue, citecolor=lred, linkcolor=luh-dark-blue}

\begin{document}
\title[Symmetry and decay of traveling waves]{Symmetry and decay of traveling wave solutions to the Whitham equation}
\author{Gabriele Bruell}
\author{Mats Ehrnstr\"om}
\author{Long Pei}
\address{Department of Mathematical Sciences, Norwegian University of Science and Technology, 7491 Trondheim, Norway.}
\email{gabriele.bruell@math.ntnu.no}
\address{Department of Mathematical Sciences, Norwegian University of Science and Technology, 7491 Trondheim, Norway.}
\email{mats.ehrnstrom@math.ntnu.no}
\address{Department of Mathematical Sciences, Norwegian University of Science and Technology, 7491 Trondheim, Norway.}
\email{long.pei@math.ntnu.no}

\subjclass[2010]{35Q53, 35B06, 35B40, 35S30, 45K05}
\keywords{Nonlocal equation; solitary solutions; symmetry; exponential decay.}

\begin{abstract}
This paper is concerned with decay and symmetry properties of solitary-wave solutions to a nonlocal shallow-water wave model. An exponential decay result for supercritical solitary-wave solutions is given. Moreover, it is shown that all such solitary-wave solutions are symmetric and monotone on either side of the crest. The proof is based on the method of moving planes. Furthermore, a close relation between symmetric and traveling-wave solutions is established.
\end{abstract}
\maketitle

\section{Introduction}

The dynamics of water waves for an inviscid perfect fluid are described by the Euler equations, complemented with suitable boundary conditions. Due to the intricate character of this system, a rigorous mathematical study of its solutions is challenging and it is one aim in the analysis of water waves to derive model equations which capture as many as possible of the phenomena displayed by water waves. In the context of irrotational, small-amplitude, shallow-water waves, it is well-known that the \emph{Korteweg--de Vries equation} (KdV),
\begin{equation}
\label{KdV}
\eta_t + \frac{3}{2}\frac{c_0}{h_0}\eta\eta_x +c_0\eta_x+\frac{1}{6} c_0h_0^2 \eta_{xxx}=0,
\end{equation}
can be rigorously deduced as a consistent approximation to the Euler equations \cite{Lannes}. Here, $\eta(t,x)$ describes the surface displacement from an undisturbed flow over a flat bottom at time $t\in[0,\infty)$ and spatial position $x\in \R$. The constant $c_0:=\sqrt{gh_0}$ is the limiting long-wave speed, $h_0$ is the undisturbed fluid depth and $g$ denotes the gravitational constant of acceleration. Equation \eqref{KdV} may be equivalently expressed in nonlocal form as
\[
\eta_t +\frac{3}{2}\frac{c_0}{h_0}\eta\eta_x +\mathcal{F}^{-1}\left(c(\xi)\right)*\eta_x=0,
\]
where $\mathcal{F}^{-1}$ denotes the inverse (spatial) Fourier transform, and 
\begin{equation*}
\label{dispersion KdV}
c(\xi):= c_0-\frac{1}{6}c_0h_0^2\xi^2
\end{equation*}
is the dispersion relation of the KdV equation. Noticing that $c$ is a second-order approximation of the exact dispersion relation of the linearized Euler equations,
\[
m_{h_0}(\xi):=\left(\frac{g\tanh(\xi h_0)}{\xi}\right)^{\frac{1}{2}}= c_0-\frac{1}{6}c_0h_0^2\xi^2 + O(\xi^4),
\]
G.~B. Whitham \cite{W67} suggested what is today termed the \emph{Whitham equation},
\begin{equation}
\label{dimW}
\eta_t + \frac{3}{2}\frac{c_0}{h_0}\eta\eta_x +\mathcal{F}^{-1}\left(m_{h_0}\right)*\eta_x=0,
\end{equation}
as an alternative to the KdV equation. Here, $K_{h_0}:=\mathcal{F}^{-1}\left(m_{h_0}\right)$ is the integral kernel corresponding to a (genuinely) nonlocal Fourier multiplier operator with symbol \(m_{h_0}\). This approach of \emph{dispersion improving} is often applied to improve the modeling aspects of fluid dynamics equations \cite{Lannes}, as it weakens the role of dispersion towards that of the full Euler equations. Equation \eqref{dimW} can also be obtained directly from the Euler equations via an exponential scaling \cite{Kalisch}. From a consistency point of view, the equation \eqref{dimW} is neither a better nor a worse model than the KdV equation: their solutions both approximate shallow-water, small-amplitude gravity water-wave solutions of the Euler equations to the same order on appropriate time scales \cite{Lannes}. As described below, the Whitham equation \eqref{dimW} however has the property of capturing several of the mathematical \emph{features} of the Euler equations, that the KdV equation does not (including nonlocality, break-down of solutions, modulational instability and highest waves).  

\medskip

The purpose of the present paper is to analyze geometric properties of solitary-wave solutions to the Whitham equation. We will show that the Whitham equation captures various characteristics of solitary solutions to the Euler equations. In the same physical setting as ours, it was shown in \cite{Craig} that any irrotational solitary gravity wave with supercritical\footnote{A wave speed exceeding \(\sqrt{gh_0}\).} wave speed is positive (a wave of elevation) and symmetric with one wave crest from which the surface decreases monotonically.  We confirm these properties for the  Whitham equation\footnote{The positivity of supercritical solutions was established in \cite{EW}.}.
Furthermore, we
address the relation between a priori symmetry and steadiness of solutions of \eqref{dimW}. As first established in \cite{EHR}, for the Euler equations as well as for a range of dispersive model equations, a priori symmetry of (time-dependent) solutions implies their being steady solutions. It turns out that this property is preserved by the Whitham equation, despite the principle in \cite{EHR} being a local one (and the Whitham equation being inherently nonlocal).

\medskip

A few words on the Whitham equation. It is straightforward \cite{EEP} to prove that \eqref{dimW} is locally well-posed in classical energy spaces $H^{\frac{3}{2}+}$, for both localized and periodic initial data, although the data-to-solution map is not uniformly continuous \cite{Arn16}. Small KdV-like solitary waves exist as constrained minimizers of a natural Hamiltonian \cite{EGW}.  Small and large periodic traveling waves connect to a global analytic curve \cite{EK}, which contains at its end a highest, cusped, wave \cite{EW} (such shapes appear also in the water wave problem, see \cite{C01Edge, H08}). The periodic waves exhibit modulation instability, as confirmed both numerically \cite{CKKS} and analytically \cite{HJ15}. The Whitham equation also allows for finite-time wave breaking in the sense of bounded surface profiles with unbounded slopes  \cite{NS, CE, Hur}. Finally, both numerical data and wave-channel experiments indicate modeling advantages of the Whitham equation when compared to the KdV \cite{BKN, TKCO16}, the Saint-Venant and the Serre equations \cite{CG}, when either short or large waves are considered.
\medskip

Our paper is outlined as follows. While Section \ref{preliminaries} only contains some short preliminaries, Section \ref{decay} is devoted to the decay of supercritical solitary-wave solutions of the Whitham equation. Inspired by the classical paper \cite{bona1997decay} on decay of solitary waves by Bona and Li, we prove that any such solution decays exponentially fast. 
In contrast, Sections \ref{TS} and \ref{ST} are concerned with the relation between traveling and symmetric wave solutions to the Whitham equation. The main result in Section \ref{TS} states that any supercritical solitary wave tending to zero at infinity is symmetric with exactly one crest from which the surface is decreasing. This result is proved by applying (a very weak form of) the \emph{method of moving planes}, which goes back to Aleksandrov \cite{Alk} and Serrin \cite{Serr} in 1962 and 1971, respectively.
 While our method is most closely related to the work by Chen, Li and Ou \cite{CLO}, and our setting to that of Craig \cite{Craig} on irrotational solitary gravity water waves, we draw some inspiration also from \cite{Hs} and \cite{CEW}. Since the method of moving planes relies upon the maximum principle, we formulate a \emph{touching lemma} for the nonlocal form of our equation, comparable to the strong maximum principle for elliptic equations.   
 In Section \ref{ST} we turn to the time-dependent Whitham equation and establish that any classical, symmetric solution, which is unique with respect to initial data is a traveling-wave solution. Although the proof has been modified to fit the nonlocal character of the Whitham equation, the result is inspired by a principle first developed in \cite{EHR}, and later used for example in \cite{K1, addAnna} (a more general approach towards such principles is in preparation \cite{BEGP}).

\section{Preliminaries}
\label{preliminaries}
To begin with, let us reformulate the Whitham equation in a normalized form as
\begin{equation}
\label{Wc}
u_t +2uu_x + K*u_x=0,
\end{equation}
where
\[
K = \mathcal{F}^{-1}\left( m \right) \qquad \mbox{and} \qquad  m(\xi) = \left(\frac{\tanh{\xi}}{\xi}\right)^\frac{1}{2}.
\]
The function $m$ is the symbol for the Fourier multiplier operator associated with the kernel $K$. 
We normalize the Fourier transform $\mathcal{F}$ of a function $f\in L_1(\R)$ to be
\[
\mathcal{F}(f)(\xi) = \int_\R f(x)e^{-ix\xi}\, dx
\]
so that the inverse Fourier transform takes the form $\mathcal{F}^{-1}(f)(\xi)= \frac{1}{2\pi}\mathcal{F}(f)(-\xi)$. 
Note that the Fourier transform on the space of Schwartz functions can be generalized by duality to a continuous isomorphism $\mathcal{F}: \mathcal{S}^\prime(\R) \to \mathcal{S}^\prime(\R)$ on the space of tempered distributions on $\R$. 
The Fourier multiplier $m$, which represents the phase speed of the linearized Euler equations, is smooth, even, and strictly decreasing on $(0,\infty)$ with $m(|\xi|)\eqsim |\xi|^{-\frac{1}{2}}$ for $|\xi|\geq 1$, attaining its maximum at $m(0)=1$. As an immediate consequence we deduce that the kernel $K$ belongs to $L_1(\R)$, is even and singular at the origin. Moreover, the analysis in \cite{EW} confirms that $K$ is positive, smooth away from the origin and completely monotone on $(0,\infty)$.

\medskip

Addressing traveling-wave solutions to \eqref{Wc}, the usual ansatz $u(x,t)= \phi(x-ct)$, where $c>0$ denotes the speed of a right-propagating wave, allows the Whitham equation to take the form
\[-c \phi_x + 2\phi\phi_x+K*\phi_x=0.\]
Integrating the above equation yields
\begin{equation}\label{steadyWhitham}
-c \phi+\phi^2 +K*\phi = B,
\end{equation}
where $B\in\R$ is an integration constant. By the Galilean change of variables
\[\phi \mapsto \phi+\gamma,\qquad c\mapsto c+2\gamma,\qquad B\mapsto B+\gamma(1-c-\gamma),\]
 we may without loss of generality consider $B=0$ in  \eqref{steadyWhitham}. 
  This choice corresponds to a solution with possibly different speed and elevation, but the form of solutions remains intact. Thus, we are left with investigating
 \begin{equation}\label{DsteadyWhitham}
-c \phi+\phi^2 +K*\phi = 0.
\end{equation}
\emph{
Throughout this paper we mean by a solution to the steady Whitham equation a real-valued, continuous, and bounded function that satisfies \eqref{DsteadyWhitham} pointwise.} Also, $ \lesssim $ and $\gtrsim$ denote that the inequality holds true up to multiplication by a positive constant. If in addition the constant depends on a parameter $p$, we write $\lesssim_{p}$ and $ \gtrsim_{p}$, respectively. Similarly, the shorthand notation $ \eqsim $ is used if both $\lesssim $ and $ \gtrsim $ hold, and $ \eqsim_{p}$ is defined accordingly.
Sometimes the notation $C=C(\cdot,\cdot,\cdots)$ is used to emphasize the dependence of a positive constant $C$ on particular parameters. 
\bigskip

\section{Decay of solitary solutions}
\label{decay}

In this section we examine the rate of decay of solitary solutions to \eqref{DsteadyWhitham}.  \emph{It is assumed that the waves are supercritical, meaning that the normalized wave speed satisfies \(c > 1\).} This is a natural assumption for gravity water waves, and in line with the current existence theory for solitary waves for the Whitham equation \cite{EGW}.
We prove that any solitary solution tending to zero at infinity decays exponentially fast. 
This is achieved by rewriting the steady Whitham equation in the form
\begin{equation}\label{convolution equation}
\phi\left(c-\phi\right)=H_c*\phi^2,
\end{equation}
where $H_c=\mathcal{F}^{-1}(\frac{m}{c-m})$, and investigating the integral kernel $H_c$. Decay properties of equations having the form 
\begin{equation*}\label{BL eq}
\phi=H*G(\phi)
\end{equation*}
are rigorously studied in the classical paper \cite{bona1997decay} by Bona and Li. In \cite{bona1997decay} the authors prove that any bounded solution $\phi$, tending to zero at infinity, decays at a rate which depends on the decay properties of the integral kernel $H$, provided $G$ satisfies a certain growth condition. 
More precisely, it is shown that if there exists $\sigma>0$ such that $e^{\sigma |\cdot|}H \in L_2(\R)$, then $\phi$ decays exponentially. Let us denote the symbol of the integral kernel by $h:=\mathcal{F}(H)$.
In view of Paley--Wiener theory, the condition on $H$ guaranteeing exponential decay of $\phi$ requires $h$ to be analytic in a horizontal complex strip enclosing the real axis.
It is then reasonable to expect that in general the lack of smoothness of the symbol $h$ yields a loss of the exponential decay property. As for the Benjamin--Ono equation, having symbol $h(\xi)=\frac{1}{1+|\xi|}$, it is known that the only solutions on $\R$ which tend to zero at infinity have quadratic decay, see \cite{AT}. A generalized Benjamin--Ono equation is studied in \cite{MM} and an algebraic decay result for solitary solutions is presented employing the regularity and (algebraic) decay of the associated kernel.
In \cite{BS} a qualitatively similar result is shown for steady solutions of the generalized Kadometsev--Petviashvili equation, whose symbol exhibits finite smoothness, too. Moreover, in \cite{BS} the authors confirm the optimality of decay. Further contributions relating finite smoothness of the symbol to algebraic decay can be found for instance in \cite{CGR,CN}.
The steady Whitham equation \eqref{convolution equation} satisfies the growth condition claimed in \cite{bona1997decay}. 
However, it can be easily seen that the kernel $H_c$  does not belong to $L_2(\R)$\footnote{its Fourier transform, given by $\frac{m}{c-m}$, is not bounded in $L_2(\R)$.}. Hence, there is no chance that $e^{\sigma |\cdot|}H_c \in L_2(\R)$ for any $\sigma>0$.
Though \eqref{convolution equation} does not fall into the frame of \cite{bona1997decay}, we prove that the kernel $H_c$ decays exponentially fast. Using then similar arguments, it is shown that any supercritical solitary  solution tending to zero at infinity decays with at least the same rate as the kernel $H_c$.

\medskip

\subsection{The kernel \(H_c\)}
It is clear that the kernel $H_c$ is even and singular at the origin. 
 In the following we shall prove that \(H_c(x) \eqsim |x|^{-\frac{1}{2}}\) for \(|x| \leq 1\) and that it decays exponentially fast, the rate of decay being (increasingly) dependent on $c$. Consequently, $|\cdot|^{\alpha}H_c(\cdot)\in L_{p}(\mathbb{R})$ for $p\in[1,\infty]$ if $\alpha >\frac{1}{2}-\frac{1}{p}$.
  Furthermore, the kernel \(H_c\) is shown to be positive and monotonically decreasing on the positive half-line.

\begin{lem}\label{lemma:analyticity}
Let \(\delta \in (0,\frac{\pi}{2})\). Then $m$ is analytic in the strip $|\Ima z|\leq \delta$. There, one has
\(
|m(z)| \leq \sqrt{\frac{\tan{\delta}}{\delta}}
\)
and $\sup_{|y| \leq \delta} \|m^{\prime}(\cdot + iy)\|_{L_2(\R)} \lesssim 1$.
\end{lem}

\begin{proof}
The function $m^2$ is holomorphic outside of $i \pi  (\frac{1}{2} + \Z)$. 
In addition one has that $m^2(z)=0$ only when $z\in i\pi \Z $, so we may take the square root and obtain that $m$ is holomorphic in the strip \(  |\Ima z|<\frac{\pi}{2}\). Fix $\delta \in (0,\frac{\pi}{2})$. For $z=x+iy$, we have that
\begin{align*}
|m(z)|^4 &= \frac{|\tanh(z)|^2}{|z|^2}= \frac{|e^z-e^{-z}|^2}{|e^z+e^{-z}|^2|z|^2} =\frac{e^{2x}+e^{-2x}-2\cos(2y)}{\left(e^{2x}+e^{-2x}+2\cos(2y) \right)(x^2+y^2)}.
\end{align*}
 This expression is uniformly bounded for $|y|\leq \delta$, where it takes its maximum at $z=i\delta$. Thus,
\(
|m(z)| \leq |m(i\delta)| = \sqrt{\frac{\tan{\delta}}{\delta}} 
\)
for  \(|\Ima z|\leq \delta\).
Note that the derivative of $m$ is odd, whence \(m^\prime(0) = 0\),  and one has 
\[
m^\prime (x) = \frac{x\sech^2{x} -\tanh{x}}{2 x^\frac{3}{2} \sqrt{\tanh{x}}}, \qquad x > 0.
\]
Since \(\tanh(x) \to 1\) as \(x \to \infty\) and \(\sech(x) \lesssim e^{-|x|}\), it follows that \(m^\prime \in L_2(\R)\). With \(z = x + iy\), one furthermore calculates
\begin{align*}
|m^\prime(z)|^{2} &= \frac{|4z - (e^{2z}-e^{-2z})|^2}{|z|^3|e^{2z}-e^{-2z}||e^z+e^{-z}|^2}\\
&= \frac{\cosh{4x} -\cos{4y}+8\left( x^{2}+y^{2} +  x  \sinh{2x}\cos{2y}+ y  \sin{2y} \cosh{2x}\right)}{2 (x^{2}+y^{2})^{\frac{3}{2}}\big( \cosh{2x} +\cos{2y}\big)\big( \cosh{4x} -\cos{4y}\big)^{\frac{1}{2}}}. 
\end{align*}
For $(x,y) \in [-M,M] \times [-\delta, \delta]$ we know that this expression is uniformly bounded. Fix \(M \gg 1\) such that \(\cosh(x) \geq \sinh(x) \gg 1\). Then for \(x > M\) one has
\begin{equation}\label{decay of l(z)}
|m^\prime(x + iy)|^{2}  \lesssim x^{-3},
\end{equation}
uniformly for \(|y| \leq \delta\). Thus \(\{m^\prime(\cdot + iy)\}_{|y| \leq \delta}\) is  bounded in \(L_2(\R)\).
\end{proof}

Now, let
\begin{equation}\label{g}
g(x) = \left(\frac{m(x)}{c-m(x)}\right)^{\prime},
\end{equation}
so that $xH_c(x)=-i\mathcal{F}^{-1}(g)(x)$. Instead of working directly with the kernel \(H_c\), we show that  $x\mapsto e^{\delta_c|x|}xH_c(x)$ belongs to $L_2(\R)$ for some constant $\delta_c>0$ depending on $c$. Here, we apply Paley--Wiener theory to the function \(g\). In the following lemma, the constant $\delta_c $ can be chosen to be increasing in $c$. As \(c \to 1\) from above, one necessarily has \(\delta_c \to 0\).

\begin{lem}\label{Kexp}  
For any given $c>1$ there exists $\delta_c\in (0,\frac{\pi}{2})$ such that
\[
x\mapsto e^{\delta_c|x|}xH_c(x)\quad \mbox{belongs to} \quad L_2(\R).
\]
\end{lem}

\begin{proof}
Fix $c>1$, and let $g$ be as in \eqref{g}. We first find a constant $\delta_c$ such that $g$ is analytic in  $|\Ima z|\leq\delta_c$ with \(\sup_{|y|\leq\delta_c}\|g(\cdot +iy)\|_{L_2(\R)} \lesssim 1\). In view of Lemma~\ref{lemma:analyticity}, and since  $\lim_{\delta\to 0}\sqrt{\frac{\tan{\delta}}{\delta}} =1$, there exists $\delta_c \in (0,\frac{\pi}{2})$ such that $\sup_{| \Ima z| \leq \delta_c} |m(z)|<c$. Hence, $(c-m)^{-2}$ is holomorphic and bounded in the same strip.  We already know that $m^\prime$ is analytic in $|\Ima z|\leq\delta_c < \frac{\pi}{2}$ and uniformly $L_2(\R)$-bounded for all $|y|\leq\delta_c$. Consequently, $g$  is analytic there, too, with
\[
\sup_{|y| \leq \delta_c} \|g(\cdot +iy)\|_{L_2(\R)} \lesssim_c \sup_{|y| \leq \delta_c}   \| m^\prime(\cdot +iy)\|_{L_2(\R)}\lesssim_c 1.
\]
The result is now a direct consequence of Paley--Wiener theory.
One then has
\[e^{\delta |\cdot|}\mathcal{F}(g) \in L_2(\R).\]
Recalling that $\mathcal{F}^{-1}(g)(x)=ixH_c(x)$, we deduce the asserted decay for \(H_c\).
\end{proof}

\begin{lem}\label{Hm near 0}
For \(|x| \leq 1\),  one has $H_c(x) \eqsim |x|^{-\frac{1}{2}}$.
\end{lem}

\begin{proof}
Note first that $g$ is odd. Thus, for \(x > 0\),
\begin{align*}
	xH_c(x) &= -i\mathcal{F}^{-1}(g)(x)\\
	&=-\displaystyle{\int_\R} g(\xi)\sin{x\xi} \,\dd\xi\\
	&=-\displaystyle{\int_\R} g\left(\frac{s}{x}\right)\frac{\sin{s}}{x}\,\dd s\\ 
	&=-\sqrt{x}\displaystyle{\int_0^\infty} \frac{\sin (s)}{s^{\frac{3}{2}}} \frac{1}{(c-m\left(\frac{s}{x}\right))^2}f\left(\frac{s}{x}\right)\,\dd s,
\end{align*}
where
\[
f(\xi) = \frac{\xi \sech^2{\xi}-\tanh{\xi}}{\sqrt{\tanh{\xi}}}.
\]
Since $c>1 \geq m$, the function $\frac{1}{(c-m)^2}f$ is bounded on $\R$ and tends to \(-\frac{1}{c^{2}}\) at infinity. By dominated convergence, we obtain that
\[
\lim_{x\searrow 0}c\sqrt{x}H_c(x)= \frac{1}{c^2}\displaystyle{\int_0^\infty}\frac{\sin (s)}{s^{\frac{3}{2}}} \, \dd s= \frac{\sqrt{2\pi}}{c^2}.
\]
The statement now follows from $H_c$ being even and continuous outside of the origin.
\end{proof}

A more detailed analysis of the function $g$ in \eqref{g} yields that $H_c$ decays not only exponentially in $L_2(\R)$ but also pointwise, although the former is enough to guarantee the exponential decay of solitary solutions to the Whitham equation.

\begin{prop}\label{kernel L infinity decay}
Let $\delta_c>0$ be as in Lemma \ref{Kexp} and $0<\delta<\delta_c$. Then
\[H_c(x) \lesssim e^{-\delta |x|}\qquad \mbox{for}\quad |x|\geq 1.\]
\end{prop}
\begin{proof}
It is immediate from Lemma \ref{Kexp} that $x\mapsto e^{\delta |x|}x^{2}H_c(x)$ belongs to $L_2(\R)$ for any $\delta \in [0,\delta_c)$. In view of the product rule its weak derivative is also bounded in $L_2(\R)$ provided that $e^{\delta |\cdot|}(\cdot)^{2}H^\prime(\cdot) \in L_2(\R)$. Similar as before we achieve the latter regularity by applying Paley--Wiener theory to 
\[k(x):=\mathcal{F}^{-1}\left((\cdot)^2H_c^\prime(\cdot)\right)(x)=i \left(x\frac{m(x)}{c-m(x)} \right)^{\prime\prime}= 2g(x)+xg^\prime(x),\]
where $g$ is defined in \eqref{g}. As in the proof of Lemma \ref{Kexp} it is a consequence of Lemma \ref{lemma:analyticity} and $c>1$, that $k$ is analytic in the strip $|\Ima z|\leq \delta$. In consideration of $g$ being uniformly $L_2(\R)$-bounded in the same strip, it remains to show that
 \begin{equation}\label{eq: g prime}
\sup_{|y|\leq\delta}\|(\cdot+iy)g^\prime(\cdot+iy)\|_{L_2(\R)}\lesssim_c 1.
\end{equation}
Consider  
\begin{equation*}
 m^{\prime\prime}(x)=\frac{\frac{3}{4}\tanh^{2}(x)-\sech^{2}(x)\tanh^{2}(x)x^{2}-\frac{1}{2}x\tanh(x)\sech^{2}(x)-\frac{1}{4}x^{2}\sech^{4}(x)}{x^{\frac{5}{2}}\tanh^{\frac{3}{2}}(x)}.
\end{equation*}
Since $|\tanh(x+iy)|\lesssim 1$, $|\sech(x+iy)|\lesssim e^{-|x|}$ when \(  |y|<\frac{\pi}{2}\), we have that
\begin{equation*}
|(x+iy) m^{\prime\prime}(x+iy)|\lesssim_{c}|x|^{-\frac{3}{2}}
\end{equation*}
uniformly for $|y|\leq \delta$ and $|x|>M$ if $M$ is chosen large enough.  Hence, $\{(\cdot+iy)m^{\prime\prime}(\cdot+iy)\}_{|y|\leq \delta}$ is bounded in $L_2(\R)$. In view of \eqref{decay of l(z)} not only $g$, but also $(\cdot)g^2(\cdot)$ is uniformly $L_2(\R)$-bounded within the strip $|\Ima z|\leq \delta$. Due to $c-m$ being bounded from below and above and
\[xg^\prime(x)=x\frac{m^{\prime\prime}(x)}{(c-m(x))^2}+2xg^2(x)(c-m(x))\]
 we conclude that \eqref{eq: g prime} holds true.
Eventually, Paley--Wiener theory implies that 
\begin{equation*}\label{H prime expo L2}
x\mapsto e^{\delta|x|}x^{2}H^{\prime}(x)\in L_{2}(\R)
\end{equation*}
and thus $e^{\delta|\cdot|}(\cdot)^{2}H(\cdot) \in H^1(\R)$. The Sobolev embedding $H^1(\R)\subset L_\infty(\R)$ ensures that $H_c(x)\lesssim e^{-\delta|x|}$ for $|x|\geq 1$.
\end{proof}

As a direct consequence of Lemmata~\ref{Kexp},~\ref{Hm near 0} and Proposition \ref{kernel L infinity decay} we obtain the following weighted $L_p(\R)$ integrability of $H_c$.

\begin{cor}\label{LPproperty of kernel}
One has $|\cdot|^{\alpha}H_c(\cdot)\in L_{p}(\mathbb{R})$ for $p\in[1,\infty]$  if and only if $\alpha >\frac{1}{2}-\frac{1}{p}$. In particular,  \(H_{c}\in L_{p}(\mathbb{R})\) exactly for \(p \in [1,2)\).
\end{cor}

It remains to show that $H_c$ is monotonically decreasing on $(0,\infty)$. To that aim, we shall need the concept of complete monotonicity: a smooth function \(f \colon (0,\infty) \to \R\) is said to be \emph{completely monotone} if 
\[
(-1)^n f^{(n)}(x) \geq 0,
\]
for all \(x > 0\) and all \(n \in \N_0\). From \cite{EW}, we have the following result:

\begin{satz}[\cite{EW}, Proposition 2.18]\label{prop:completely monotone Fourier}
Let \(f\) and \(h\) be two functions satisfying \(f(\xi)=h(\xi^2)\).
Then \(f\) is  the Fourier transform of an even, integrable function such that
\(\mathcal{F}^{-1} (f)(\sqrt{\cdot})\) is completely monotone if and only if \(h\) is completely monotone with \(\lim_{\lambda \searrow 0} h(\lambda)<\infty\) and \(\lim_{\lambda \to \infty} h(\lambda)=0\). In this case,  \(\mathcal{F}^{-1} (f)\) is smooth and monotone outside of the origin.
\end{satz}

\begin{cor}
The integral kernel \(H_c\) is positive, smooth, and monotonically decreasing on the positive half-line $(0,\infty)$.
\end{cor}

\begin{proof}
Let  $h(x) :=\frac{m(\sqrt{x})}{c-m(\sqrt{x})}$. Then $\lim_{x\to 0}h(x) = \frac{1}{c-1}$ and $\lim_{x\to\infty}h(x)=0$. Thus, in view of Proposition~\ref{prop:completely monotone Fourier}, it is sufficient to prove that the function  $h$ is completely monotone. Let $n(x) =m(\sqrt{x})$ and consider $h = \frac{n}{c-n}$. In \cite{EW}  it is proved that $n$ is completely monotone. By combining this with Leibniz's rule we obtain for $x\in (0,\infty)$ that
\begin{equation*}
\begin{split}
&(-1)^m h^{(m)}(x)\\
&=\sum_{k=0}^{m}\binom{m}{k}(-1)^{m}n^{(m-k)}\left(\frac{1}{c-n}\right)^{(k)}(x)\\
&=(-1)^{m}n^{(m)}\left(\frac{1}{c-n}\right)(x)\\
&+\sum_{k=1}^{m}\binom{m}{k} (-1)^{m-k}n^{(m-k)}(x)\sum\frac{k!}{\prod_{j=1}^{k}(b_{j}!)} \frac{\tilde{k}!}{(c-n(x))^{\tilde{k}+1}} \prod_{j=1}^{k}\left( {\textstyle\frac{(-1)^{j}n^{(j)}(x)}{j!}} \right)^{b_{j}},
\end{split}
\end{equation*}
where the second sum is over all $k$-tuples of nonnegative integers $(b_1,..., b_k)$ satisfying the constraint \(\sum_{1\leq j\leq k} j b_j = k\) and $\tilde{k}=b_{1}+\cdot\cdot\cdot+b_{k}$.
It follows immediately that \(h\) is completely monotone, whence Proposition \ref{prop:completely monotone Fourier} implies that \(H_c\) is positive, integrable, smooth, and monotone outside of the origin.
\end{proof}

\subsection{Algebraic decay of solitary solutions}
We start with a prior result on algebraic decay, displaying the importance of the quadratic nonlinearity\footnote{More precisely, it is necessary that the right-hand side of \eqref{convolution equation} has the form $H_c*G(\phi)$, where $G(s)\lesssim |s|^\gamma$ for some $\gamma>1$ and small values of $s$, cf. also \cite{bona1997decay}.}.
In particular, we make evident that for arbitrary $l\geq 0$ a supercritical solution $\phi$ tending to zero at infinity of the steady Whitham equation \eqref{convolution equation} satisfies
\[x\mapsto|x|^{l}\phi(x)\in L_{\infty}(\R).\] 

Let us begin with a lemma guaranteeing that the term $c-\phi$ on the left-hand side of the steady Whitham equation is bounded from below and above, if $c>1$ .
\begin{lem}\label{more precise bound of phi} Let $c>1$.
Any nonzero continuous bounded solution $\phi$ to the steady Whitham equation \eqref{convolution equation} satisfies
\begin{equation*}
0<\phi<c.
\end{equation*} 
If additionally $\phi(x)\to 0$ as $|x|\to \infty $, then $\sup_{x\in\R} \phi(x)<c$. 
\end{lem}

\begin{proof}
In \cite[ Lemma 4.1]{EW} it is shown that $\inf_{x\in\R} \phi \in [0,c-1]$. Note in particular that the solution $\phi$ is nonnegative.
In view of the Whitham kernel $K$ being strictly positive, any nonzero solution to the steady Whitham equation fulfills the inequality
\[c\phi - \phi^2 = K*\phi > 0.\]
Hence, $\phi$ is bounded from below by zero and from above by $c$. Assuming that $\phi$ tends to zero at infinity, continuity implies that $\sup_{\R} \phi(x) < c$.
\end{proof}

The following theorem is the key result for algebraic decay and a modified version of \cite[Lemma 10]{MM}, where the decay properties of solitary-wave solutions to a generalized Benjamin--Ono equation is investigated. 
\begin{thm}\label{algebraic decay 1}
Let $\phi$ be a supercritical solution to the steady Whitham equation \eqref{convolution equation} and $\phi(x)\to 0$ as $|x|\to \infty$. 
Then, 
	\[x\mapsto |x|^{l}\phi(x)\in L_{q}(\mathbb{R})\] for all $q
	\in (2,\infty)$ and any $l \geq 0$. 
\end{thm}

\begin{proof}
	In view of Lemma \ref{more precise bound of phi}, there exists a constant $M\in(0,c)$ such that $\sup \phi = M$.
	Choose $p\in (1,2)$ and let $\alpha=\alpha(p)$ be a positive constant satisfying
	\begin{equation*}
	\alpha >1-\frac{1}{p}.
	\end{equation*}
	 Corollary \ref{LPproperty of kernel} guarantees that the function $(1+|\cdot|)^{\alpha}H_c(\cdot)$ is bounded in $L_{p}(\R)$. We set 
	 \begin{equation*}
	 K_{\alpha,p}:=(c-M)^{-1}\|(1+|\cdot|)^{\alpha}H_c(\cdot)\|_{L_{p}(\R)}.
	 \end{equation*} 
	  Let $q$ be the conjugate of $p$, i.e. $\frac{1}{p}+\frac{1}{q}=1$.
	 As $\phi$ is a solution to \eqref{convolution equation}, we can write
	\begin{align*}
	\phi(c-\phi)(x)&=(H_c*\phi^2)(x)\\
	&=\int_{\R}H_c(x-y)(1+|x-y|)^{\alpha}\cdot \frac{\phi^2(y)}{(1+|x-y|)^{\alpha}}\dd y
	\end{align*}
	 and obtain by  H\"{o}lder's inequality that
	\begin{equation}\label{f control 1}
	|\phi(x)|\leq K_{\alpha,p}\left(\int_{\R}\frac{|\phi^2(y)|^{q}}{(1+|x-y|)^{\alpha q}}\dd y \right)^{\frac{1}{q}}.
	\end{equation} 
	Let $l\in [0,\alpha-\frac{1}{q})$. Then $\alpha> l+\frac{1}{q}$ and we define 
	\begin{equation*}
	h_{\e}(x):=\frac{|x|^{l}}{(1+\e |x|)^{\alpha}}\phi(x)
	\end{equation*}
	 for $0<\e<1$.
	 For each $\e\in (0,1)$ fixed, the function $h_{\e}$ is bounded in $L_{q}(\R)$, by the choice of $l$ and $\phi$ being bounded.  The aim is to prove that $\{h_\e\mid \e\in(0,1)\}$ is uniformly bounded in $ L_{q}(\R)$, which implies that $\lim_{\e\to 0} h_\e = |\cdot|^l \phi$ belongs to $L_q(\R)$, by dominated convergence.
	 Since  $\phi$ tends to zero as $|x|\to \infty$, the quadratic nonlinearity provides that for every $\delta>0$ there exists  a constant $R_{\delta}>1$ such that
	 \begin{equation*}
	|\phi^2(x)|\leq \delta|\phi(x)|\qquad \mbox{for}\quad |x|\geq R_{\delta}.
	\end{equation*}
	Estimating
	\begin{equation}\label{spitting h}
	\|h_\e\|^q_{L_q(\R)}= \int_\R \left|h_\e(x)\right|^q \dd x \leq C+ \int_{|x|\geq R_\delta} \left|h_\e(x)\right|^q \dd x,
	\end{equation}
	where $C=C(R_\delta)>0$ is a constant independent of $\e$,
	we are left to study the last integral on the right-hand side of \eqref{spitting h}. 
	Let $r\in (0,q)$. Thanks to \eqref{f control 1} and H\"older's inequality we have that 
	\begin{align*}
	&\int_{|x|\geq R_\delta}|h_{\e}(x)|^{q}\dd x \leq\int_{|x|\geq R_\delta}|h_{\e}(x)|^{q-r}\bigg(\frac{|x|^{l}}{(1+\e|x|)^{\alpha}}\bigg)^{r}|\phi(x)|^{r}\dd x\\
	&\qquad\leq \int_{|x|\geq R_\delta}|h_{\e}(x)|^{q-r}\bigg(\frac{|x|^{l}}{(1+\e|x|)^{\alpha}}\bigg)^{r} K^{r}_{\alpha,p}\bigg(\int_{\R}\frac{|\phi^2(y)|^{q}}{(1+|x-y|)^{\alpha q}}\dd y\bigg)^{\frac{r}{q}}\dd x\\
	&\qquad\leq K^{r}_{\alpha,p}\bigg[\int_{|x|\geq R_\delta}|h_{\e}(x)|^{q}\dd x\bigg]^{\frac{q-r}{q}}\\
	&\qquad\quad\cdot\bigg[\int_{|x|\geq R_\delta}\bigg(\frac{|x|^{l}}{(1+\e|x|)^{\alpha}}\bigg)^{q}\left(\int_{\R} \frac{|\phi^2(y)|^{q}}{(1+|x-y|)^{\alpha q}}\dd y\right)\dd x\bigg]^{\frac{r}{q}}.
	\end{align*}
	Dividing both sides of the inequality by $\bigg[\int_{|x|\geq R_\delta}|h_{\e}(x)|^{q}\dd x\bigg]^{\frac{q-r}{q}}$ we find that\footnote{Note that the term we are dividing by vanishes if and only if $\phi=0$ everywhere in $\{|x|\geq R_\delta\}$, in which case the lemma is obviously true. }
	\begin{equation}\label{comb 1}
	\int_{|x|\geq R_\delta}|h_{\e}(x)|^{q}\dd x\leq K^{q}_{\alpha,p} \int_{|x|\geq R_\delta}\bigg(\frac{|x|^{l}}{(1+\e|x|)^{\alpha}}\bigg)^{q}\left(\int_{\R} \frac{|\phi^2(y)|^{q}}{(1+|x-y|)^{\alpha q}}\dd y\right)\dd x.
	\end{equation}
	One can then invoke Fubini's Theorem and Lemma \ref{main inequality lemma} to obtain that
	\begin{align}\label{comb 2}
	\begin{split}
	\int_{|x|\geq R_\delta}&\bigg[\bigg(\frac{|x|^{l}}{(1+\e|x|)^{\alpha}}\bigg)^{q}\int_{\R} \frac{|\phi^2(y)|^{q}}{(1+|x-y|)^{\alpha q}}dy\bigg]\dd x \\
	&=\int_{\R}|\phi^2(y)|^{q}\bigg[\int_{|x|\geq R_\delta}\frac{|x|^{lq}}{(1+\e|x|)^{\alpha q}(1+|x-y|)^{\alpha q}}\dd x\bigg]\dd y\\
	&\leq \int_{|x|\geq R_\delta}|\phi^2(y)|^{q}\frac{B|y|^{lq}}{(1+\e|y|)^{\alpha q}}\dd y\\
	&\quad+\int_{|y|<R_\delta}|\phi^2(y)|^{q}\int_{|x|\geq R_\delta} \frac{|x|^{lq}}{(1+\e|x|)^{\alpha q}(1+|x-y|)^{\alpha q}}\dd x\dd y,
	\end{split}
	\end{align}
	where $B=B(l,q,\alpha)>0$ does not depend on $\e$. Since $\alpha q>1$ and $lq<\alpha q-1$, by the choice of $l$,  the last integral in \eqref{comb 2} is bounded by a constant $C_1=C_1(\|\phi\|_\infty,R_{\delta})>0$ depending on the norm of $\phi$ and $R_{\delta}$ (but not on $\e$).
	Combining \eqref{comb 1}, \eqref{comb 2} and recalling that $|\phi^2(y)|<\delta|\phi(y)|$ for all $|y|\geq R_\delta$, we deduce that
	\begin{equation}\label{part 1 bound}
	\int_{|x|\geq R_\delta}|h_{\e}(x)|^{q} \dd x\leq K_{\alpha,p}^p \left[\delta ^qB\int_{|x|\geq R_\delta}|h_{\e}(x)|^{q}\dd x+C_1\right] .
	\end{equation}
	 Choosing $\delta$ small enough so that $K_{\alpha,p}^p\delta^qB<\frac{1}{2}$ , \eqref{part 1 bound} implies that
	\begin{equation*}
	\int_{|x|\geq R_\delta}|h_{\e}(x)|^{q}dx\leq C_{2},
	\end{equation*}
	where $C_2=C_{2}(\alpha, p, \|\phi\|_\infty,R_{\delta})>0$ is a constant which does not rely on $\e$. 
	Hence, we have shown that
\[\int_\R |h_{\e}(x)|^{q}dx\lesssim 1.\]
Letting $\e\to 0$, dominated convergence ensures that
	\begin{equation*}
	\int_\R |x|^{lq}|\phi(x)|^{q}dx\lesssim  1,
	\end{equation*}
	which implies in particular $x\mapsto|x|^{l}f(x)\in L_{q}(\R)$ for $q=\frac{p}{p-1}$ and $l\in[0,\alpha-\frac{1}{p})$. Having at hand that $\alpha$ can be chosen arbitrarily large and $p\in(1,2)$, the statement is proved.
\end{proof}

\begin{remark}\emph{
Note that the proof uses only the algebraic decay properties of $H_c$, that is $|\cdot|^\alpha H_c \in L_p(\R)$ for $p\in [1,2)$  and $\alpha>0$. It is apparent from the proof that the decay rate $l$ depends increasingly on $\alpha$.}

\end{remark}
The following algebraic decay result is an immediate consequence of the previous theorem.

\begin{cor}[Algebraic decay]\label{algebraic decay}
	Let $\phi$ be a supercritical solution to the steady Whitham equation \eqref{convolution equation} and $\phi\to 0$ as $|x|\to\infty$. Then
	 \[x\mapsto|x|^{l}f(x)\in L_{\infty}(\R),\]
	 for any $l\geq 0$.

\end{cor}
\begin{proof} Let $l\geq 0$ be arbitrary. Then, Lemma \ref{more precise bound of phi} implies that
\begin{align*}\label{N power inequality}
|x|^{l}|\phi(x)|&\lesssim  \left(|\cdot|^lH_c *\phi^2\right)(x)+\left(H_c*|\cdot|^l\phi^2\right)(x).
\end{align*}
In consideration of Theorems \ref{LPproperty of kernel} and \ref{algebraic decay 1} the assertion follows by Young's inequality. 
 \end{proof}

\subsection{Exponential decay of solitary solutions}

Relying on the arguments in \cite[Corollary 3.1.4]{bona1997decay}, we apply the exponential decay of $H_c$ to prove that for the steady Whitham equation any supercritical solution $\phi$  actually decays exponentially.

\begin{thm}[Exponential decay]\label{exponential decay 1}
Let $\delta_c>0$ denote the decay rate of $H_c$. If $\phi$ is a supercritical solution to the steady Whitham equation \eqref{convolution equation} satisfying $\phi\to 0$ as $|x|\to \infty$, then there exists $0<\nu<\delta_c$ such that
\[x\mapsto e^{\nu|x|}\phi(x)\quad \mbox{belongs to}\quad L_1(\R)\cap L_\infty(\R).\]
\end{thm}

\begin{proof}
Corollary \ref{Kexp} warrants that for every $c>1$ there exists $\delta_c\in (0,\frac{\pi}{2})$ such that 
\[x\mapsto e^{\delta_c|x|}xH_c(x)\quad \mbox{belongs to}\quad L_2(\R).\]
Together with Corollary \ref{LPproperty of kernel} we conclude that for all $0<\sigma<\delta_c$
\begin{equation}\label{K regularity}
e^{\sigma |\cdot|}H_c(\cdot)\in L_p(\R)\quad \mbox{for}\; p\in [1,2).
\end{equation}
Choose $p\in [1,2)$, let $q\in \R$ be such that $\frac{1}{p}+\frac{1}{q}=1$, and set
\[D:= \max \left\{1, \frac{\sigma}{2}\|\phi\|_{L_1(\R)}, (c-M)^{-1}\sigma^{\frac{1}{p}}\left(\frac{2}{q}\right)^{\frac{1}{q}}\|e^{\sigma|\cdot|}H_c(\cdot)\|_{L_p(\R)} \|(\cdot)\phi(\cdot)\|_\infty \right\},\]
where $M:=\sup \phi <c$.
Note that $D\geq 1$ is finite, owing to $\phi$ being a bounded solution, Corollary \ref{algebraic decay} and \eqref{K regularity}.
The main ingredient for proving exponential decay of $\phi$ is the following estimate
\begin{equation}\label{main estimate}
\|(\cdot)^l\phi(\cdot)\|_{L_1(\R)}\leq   \frac{(l+2)!D^{l+1}}{\sigma^{l+1}}\quad \mbox{for }\; l\in \N_0\quad \mbox{and}\quad 0<\sigma<\delta_c.
\end{equation}
Claim \eqref{main estimate} is proved by induction. Clearly, the statement holds true for $l=0$. Assuming that the inequality \eqref{main estimate} is satisfied for all natural numbers less or equal some $l\in \N_0$ one observes that
\begin{align}\label{estimate n1}
\begin{split}
\|(\cdot)^{l+1}\phi(\cdot)\|_{L_1(\R)}&\leq (c-M)^{-1}\|(\cdot)^{l+1}(H_c*\phi^2)(\cdot)\|_{L_1(\R)}\\
&\leq (c-M)^{-1}\sum_{j=0}^{l+1}\binom{l+1}{j}\|(\cdot)^{l+1-j}H_c(\cdot)\|_{L_1(\R)}\|(\cdot)^{j}\phi^2(\cdot)\|_{L_1(\R)},
\end{split}
\end{align}
on account of Lemma \ref{convolution lemma} and Young's inequality.
Applying H\"older's inequality to $\|(\cdot)^{l+1-j}H_c(\cdot)\|_{L_1(\R)}$ yields
\begin{align*}
\|(\cdot)^{l+1-j}H_c\|_{L_1(\R)}&=\int_\R |x^{l+1-j}H_c(x)| \dd x\leq \int_\R |x^{l+1-j}e^{-\sigma|x|}||e^{\sigma|x|}H_c(x)|\dd x\\ 
&\leq \|e^{\sigma|\cdot|}H_c(\cdot)\|_{L_p(\R)} \left(\int_\R |x|^{q(l+1-j)}e^{-q\sigma|x|}\dd x\right)^{\frac{1}{q}}\\ &=\|e^{\sigma|\cdot|}H_c(\cdot)\|_{L_p(\R)} 2^{\frac{1}{q}} \left(\int_0^\infty x^{q(l+1-j)}e^{-q\sigma x}\dd x\right)^{\frac{1}{q}}.
\end{align*}
Due to Lemma \ref{Alemma} one arrives at
\begin{align*}
\begin{split}
\|(\cdot)^{l+1-j}H_c\|_{L_1(\R)}&\leq  \|e^{\sigma|\cdot|}H_c(\cdot)\|_{L_p(\R)} 2^{\frac{1}{q}}\left( \frac{\left[q(l+1-j)\right]!}{(q\sigma)^{q(l+1-j)+1}} \right)^{\frac{1}{q}}\\
&\leq \|e^{\sigma|\cdot|}H_c(\cdot)\|_{L_p(\R)} \left(\frac{2}{q}\right)^{\frac{1}{q}}\frac{(l+1-j)!}{\sigma^{l+1-j+\frac{1}{q}}}.
\end{split}
\end{align*}
The assumption that \eqref{main estimate} holds for all natural numbers less of equal $l\in \N_0$ allows to control the second factor of the last inequality in \eqref{estimate n1} by
\begin{align}\label{estimate n3}
\|(\cdot)^{j}\phi^2(\cdot)\|_{L_1(\R)}&\leq \|(\cdot)\phi(\cdot)\|_\infty\|(\cdot)^{j-1}\phi(\cdot)\|_{L_1(\R)} \leq \|(\cdot)\phi(\cdot)\|_\infty\frac{(j+1)!D^{j}}{\sigma^{j}}
\end{align}
for $1\leq j\leq l+1$.
The combination of \eqref{estimate n1}--\eqref{estimate n3} together with the definition of $D$ yields
\begin{align*}
&\|(\cdot)^{l+1}\phi(\cdot)\|_{L_1(\R)}\leq \frac{1}{c-M}\sum_{j=0}^{l+1}\binom{l+1}{j}\|(\cdot)^{l+1-j}H_c(\cdot)\|_{L_1(\R)}\|(\cdot)^{j}\phi^2(\cdot)\|_{L_1(\R)}\\
&\quad=(c-M)^{-1}\sum_{j=0}^{l+1} \|e^{\sigma|\cdot|}H_c(\cdot)\|_{L_p(\R)} \left(\frac{2}{q}\right)^{\frac{1}{q}} \|(\cdot)\phi(\cdot)\|_\infty\frac{(l+1)!(j+1)D^{j}}{\sigma^{l+1+\frac{1}{q}}}\\
&\quad\leq\sum_{j=0}^{l+1}  \frac{(l+1)!(j+1)D^{j+1}}{\sigma^{l+1+\frac{1}{q}+\frac{1}{p}}} \\
&\quad \leq \frac{(l+3)!D^{l+2}}{\sigma^{l+2}} ,
\end{align*}
which completes the inductive step. Eventually, \eqref{main estimate} implies that
\begin{align*}
\|e^{\nu |\cdot|}\phi\|_{L_1(\R)} &= \int_\R \left|\sum_{l=0}^\infty \frac{\nu^l x^l}{l!}\phi(x) \right|\dd x\leq \sum_{l=0}^\infty \frac{\nu^l }{l!} \|(\cdot)^l\phi(\cdot)\|_{L_1(\R)}\\
&\leq \frac{D}{\sigma}\sum_{l=0}^\infty (l+1)(l+2)\left(\frac{\nu D}{\sigma}\right)^l,
\end{align*}
which converges if and only if $|\nu| <\frac{\sigma}{D}$. The boundedness of $\|e^{\nu |\cdot|}\phi\|_{L_{\infty}(\R)}$ can be proved similarly  by replacing $\|(\cdot)^l\phi(\cdot)\|_{L_1(\R)}$ by $\|(\cdot)^l\phi(\cdot)\|_{L_{\infty}(\R)}$ in \eqref{main estimate} and modifying $D$ accordingly.  Summarizing, we have that
\[e^{\nu |\cdot|}\phi\in L_1(\R)\cap L_\infty(\R)\qquad \mbox{for any }\quad 0<\nu<\frac{\sigma}{D}.\]
By definition of $\sigma$ and $D$, one observes that $\nu < \delta_c$
\end{proof}

The result above can be improved. As a matter of fact, following the lines in \cite[Corrollary 3.1.4]{bona1997decay}, one can show that any supercritical solution of the steady Whitham equation, tending to zero at infinity, decays at least at the same rate as the associated kernel $H_c$ does. In interest of keeping the present paper self contained, we include the proof.

\begin{prop}\label{exponential decay}
Let  $\delta_c>0$ be the decay rate of $H_c$. If $\phi$ is a supercritical solution to the steady Whitham equation \eqref{convolution equation} satisfying $\phi\to 0$ as $|x|\to \infty$, then
\[e^{\eta |\cdot|}\phi\in L_1(\R)\cap L_\infty(\R)\qquad \mbox{for some }\quad \eta \geq \delta_c.\]
\end{prop}

\begin{proof}
 Theorem \ref{exponential decay} ensures that there exists $0<\nu<\delta_c$ such that
\[e^{\nu |\cdot|}\phi\in L_1(\R)\cap L_\infty(\R).\] 
Moreover, recall from \eqref{K regularity} that 
\begin{equation}\label{K12 regularity}
e^{\nu |\cdot|}H_c \in L_1(\R) \quad \mbox{for any} \quad 0<\nu<\delta_c.
\end{equation}
Thanks to the  quadratic nonlinearity we can estimate
\begin{align}\label{2conv}
\begin{split}
\phi(x)e^{\nu |x|}&\leq  (c-M)^{-1}\int_\R H_c (x-y)e^{\nu|x-y|}|\phi^2(y)|e^{\nu |y|} \dd y \\
&=(c-M)^{-1} \int_\R H_c (x-y)e^{\nu|x-y|}\left(\phi(y)|e^{\frac{\nu}{2} |y|}\right)^2 \dd y \\
&= (c-M)^{-1}\left(H_ce^{\nu |\cdot|}*\left(\phi e^{\frac{\nu}{2} |\cdot|}\right)^2\right)(x),
\end{split}
\end{align}
where $M:=\sup \phi < c$.
Let $\eta:=\sup\{\nu>0\mid e^{\nu |\cdot|}\phi\in L_1(\R)\cap L_\infty(\R)\}$. The aim is to show that $\eta\geq \delta_c$. 
Assuming on the contrary that $\eta <\delta_c$, one can choose $\nu>0$ such that $\frac{\eta}{2}<\nu<\min\{\eta,\frac{\delta_c}{2}\}$. 
Considering \eqref{2conv} for $2\nu$ instead of $\nu$, Young's inequality, \eqref{2conv} and $\phi$ being bounded imply that
\begin{equation}\label{contradiction}
\begin{split}
\|\phi e^{2\nu |\cdot|}\|_{L_1(\R)}&\leq(c-M)^{-1} \| H_ce^{2\nu |\cdot|}*\phi^2 e^{\nu |\cdot|}\|_{L_1(\R)}<\infty
\end{split}
\end{equation}
and
\begin{equation}\label{contradiction2}
\begin{split}
\|\phi e^{2\nu |\cdot|}\|_{L_\infty(\R)}&\leq(c-M)^{-1}\| H_ce^{2\nu |\cdot|}*\phi^2 e^{\nu |\cdot|}\|_{L_\infty(\R)}<\infty,
\end{split}
\end{equation}
by the choice  of $\nu$ and \eqref{K12 regularity}. In view of $2\nu > \eta$, \eqref{contradiction} and \eqref{contradiction2} lead to a contradiction to the definition of $\eta$. Hence, $\eta\geq \delta_c$.
\end{proof}

\bigskip

\section{Symmetry of solitary waves}\label{TS}
The method of moving planes is employed to prove that any supercritical solution to the steady Whitham equation tending to zero at infinity is symmetric and has exactly one crest. 
Let us start by formulating a lemma, which is comparable to the strong maximum principle for elliptic equations. It is a modified version of the so called touching lemma in \cite[Lemma 4.3]{EW} and stated in a form appropriate to our purpose.
A solution $\phi$ to the steady Whitham equation \eqref{convolution equation} is called a \emph{supersolution} if 
\[\phi(c-\phi)\geq H_c*\phi^2\]
and a \emph{subsolution} if the inequality above is replaced by $\leq$.

\begin{lem}[Touching lemma on a half-space]\label{touching lemma} Let $\phi_1$ and $\phi_2$ be a super-- and a subsolution of the steady Whitham equation \eqref{convolution equation} on a subset  $[\lambda,\infty)\subset \R$, respectively, satisfying $\phi_1\geq \phi_2$ on  $[\lambda,\infty)$ and $\phi_1^2-\phi_2^2$ being odd with respect to $\lambda$, that is $(\phi_1^2-\phi_2^2)(x)=-(\phi_1^2-\phi_2^2)(2\lambda-x)$. Then either
\begin{itemize}
\item $\phi_1=\phi_2$ in $[\lambda,\infty)$, or
\item $\phi_1>\phi_2$ with $\phi_1+\phi_2<c$ in $(\lambda,\infty)$ .
\end{itemize}
\end{lem}

\begin{proof}The symmetry and monotonicity of $H_c$ allow to deduce that $H_c$ acts as a positive convolution operator on odd functions with respect to $\lambda$ on the half line $[\lambda,\infty)$. Let $f\geq 0$ on $[\lambda,\infty)$, $f(x)=-f(2\lambda-x)$ and $x\geq \lambda$. Then
\begin{align*}
H_c*f(x)&=\int_{\lambda}^{\infty}H_c(y)f(x-y) \dd y +\int_{-\infty}^\lambda H_c(x-y)f(y) \dd y \\
&=\int_{\lambda}^{\infty}H_c(x-y)f(y) \dd y +\int^{\infty}_\lambda H_c(x+y-2\lambda)f(2\lambda-y) \dd y\\
&= \int_{\lambda}^{\infty}(H_c(x-y)-H_c(x+y-2\lambda))f(y)\dd y,
\end{align*}
where the last equality holds thanks to $f$ being odd with respect to $\lambda$. In view of $H_c$ being symmetric and monotonically decreasing on $(0,\infty)$, we obtain that
\begin{equation*}
H_c*f(x)\geq 0 \qquad \mbox{for all }\quad x\geq \lambda.
\end{equation*}
In particular, $H_c*f>0$ or $f=0$ on $(\lambda,\infty)$. 
Assume that $\phi_1$ and $\phi_2$ are super-- and subsolutions to the steady Whitham equation, respectively, $\phi_1\geq \phi_2$ for all $x\geq \lambda$ and  $\phi_1^2-\phi_2^2$ is odd with respect to $\lambda$, that is $\phi_1^2-\phi_2^2$ plays the role of $f$ above. Then,
\begin{align*}
(c-(\phi_1+\phi_2))(\phi_1-\phi_2)\geq H_c*(\phi_1^2-\phi_2^2)>0
\end{align*}
for all $x>\lambda$ unless $\phi_1=\phi_2$ on $[\lambda,\infty)$.
\end{proof}

\begin{cor} Let $\phi$ be a solution to the steady Whitham equation \eqref{convolution equation} and $\phi_\lambda(\cdot):=\phi(2\lambda-\cdot)$ be its refection about some $\lambda\in \R$. If $\phi\geq \phi_\lambda$ on $[\lambda,\infty)$ and there exists a point $x>\lambda$ where $\phi$ and $\phi_\lambda$ touch, that is $\phi(x)=\phi_\lambda(x)$, then $\phi=\phi_\lambda$.
\end{cor}

\begin{proof}
Let $\phi$ be a solution of the steady Whitham equation, then so is $\phi_\lambda$ due to the symmetry of $H_c$. Noticing that $\phi^2-\phi_\lambda^2$ is odd with respect to $\lambda$, the assertion is an immediate consequence of Lemma \ref{touching lemma}.
\end{proof}

The method of moving planes is applied to confirm that any supercritical solution tending to zero at infinity of the steady Whitham equation \eqref{convolution equation}
has exactly one crest about which it is symmetric.
The proof is inspired by  \cite{CLO}, where the authors establish the symmetry of positive solutions belonging to $L_{\frac{n+\alpha}{n-\alpha}, \mathrm{loc}}(\R^n)$ of
\begin{align}
\label{CLO eq}
u=L*u^{\frac{n+\alpha}{n-\alpha}},
\end{align}
where $n$ is the space dimension, $\alpha\in(0,n)$ and $L(x):= |x|^{\alpha-n}$. If $\alpha =\frac{1}{3}$ and the space dimension $n=1$, equation \eqref{CLO eq} reads 
\[u=L*u^{2}\]
with $L(x)=|x|^{-\frac{2}{3}}$, which displays a structural similarity to \eqref{convolution equation}. 
Assuming that $\phi$ tends to zero at infinity, our proof is less intricate than in \cite{CLO}, where the authors do not assume any asymptotic behavior of the solution but apply the method of moving planes to a Kelvin-type transform instead.
Since the nonlocal operator $H_c$ of the steady Whitham equation corresponds to an \emph{inhomogeneous} kernel function a Kelvin-type transform is not appropriate in our case. It is worth mentioning that in \cite{CLL, CLLY} the authors generalize the result in \cite{CLO} and establish maximum principles for a class of nonlocal equations which originate from the fractional Laplace operators.

\medskip

In accordance to \cite{CLO}, we define the open sets
\[\Sigma_\lambda :=\{x\in \R \mid x> \lambda \}\quad \mbox{and}\quad \Sigma_\lambda ^-:=\{x\in \Sigma_\lambda  \mid \phi(x)<\phi_\lambda(x)\},\]
where $\phi_\lambda(\cdot):= \phi(2\lambda-\cdot)$ is the reflection of $\phi$ about the axis $x=\lambda$.

\begin{center}
\begin{figure}[h!]
\begin{tikzpicture}[scale=0.8]
\draw[->] (-6,0)--(6,0) node[right]{$x$};
		 \coordinate (dstart) at (-6,0.8);
    \coordinate (d1) at (-5,1);
    \coordinate (d3) at (0,3.5);
    \coordinate (d4) at (2.5,0.8);
    \coordinate (d5) at (4,1.5);
    \coordinate (dend) at (6,1);
    \draw [color=black] (dstart) to[out=0,in=180] (d1) to[out=0,in=180] (d3)to[out=0,in=180] (d4)  to[out=0,in=200] (d5) to[out=10,in=160] (dend);				
	\coordinate (start) at (-1.5,3);
	\coordinate (e1) at (2.5,1);
	\coordinate (e2) at (3.5,0.8);
	\coordinate (e3) at (6,0.6);
	\draw [dashed] (start) to[out=-30,in=180] (e1) to[out=0,in=180] (e2) to[out=0,in=180] (e3);							
	\draw [dashed] (-1.5,3) -- (-1.5,0) node[below]{$\lambda$};			
	\draw [->, bend angle=45, bend right]  (-2,1.5) to (-1,1.5);
	\draw[very thick] (1.8,0)--(3.2,0);					
	\end{tikzpicture} 
	\caption{Sketch of an arbitrary supercritical solution to \eqref{convolution equation} tending to zero at infinity. The dashed curve is the reflection of the wave around $\lambda$ and the bold line on the {$x$-axis} represents the set $\Sigma_{\lambda}^-$.}
\end{figure}
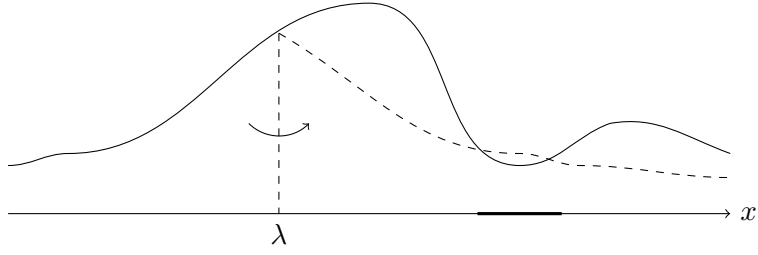
\end{center}

The aim is to prove that there exists $\lambda_0\in\R$ such that $\phi$ is symmetric about $x=\lambda_0$, that is $\phi(x)=\phi_\lambda(x)$ for all $x\in\R$, and moreover that $\phi$ has only one crest, which is then located at $x=\lambda_0$.
As a first step, it is shown that there exists $\lambda\in \R$ far enough to the left, so that the open set $\Sigma_\lambda^-$ is empty. An easy calculation analog to the one in the proof of Lemma \ref{touching lemma} shows that
\begin{align*}
\begin{split}
c(\phi(x)-\phi_\lambda(x))&=\int_{\Sigma_{\lambda}}\big(H_c(x-y)-H_c(2\lambda -x-y)\big)(\phi^{2
}(y)-\phi_\lambda^2(y))\dd y \\
& \quad+\phi^{2}(x)-\phi_\lambda^2(y).
\end{split}
\end{align*}
Let $x\in \Sigma_\lambda^-$. Then
\begin{align*}
0&<c(\phi_\lambda(x)-\phi(x))\\
&\leq \int_{\Sigma_{\lambda^-}}\big(H_c(x-y)-H_c(2\lambda -x-y)\big)(\phi_\lambda^2(y)-\phi^{2}(y))\dd y +\phi_\lambda^2(x)-\phi^{2}(x)\\
&\leq 2\int_{\Sigma_{\lambda^-}}H_c(x-y)\phi_\lambda(y)(\phi_\lambda(y)-\phi(y))\dd y +\phi_\lambda^2(x)-\phi^{2}(x)\\
&= 2\big(H_c*\phi_\lambda(\phi_\lambda-\phi)\big)(x) + \phi_\lambda^2(x)-\phi^{2}(x).
\end{align*} 
By Young's inequality we arrive at
\begin{equation}\label{lqnorm}
\left\|\phi_\lambda-\phi\right\|_{L_\infty({\Sigma^{-}_{\lambda}})}\leq \frac{2}{c}\|\phi\|_{L_{\infty}\left({\left(\Sigma_\lambda^-\right)^*}\right)}\left( \|H_c\|_{L_{1}({\R})}+1\right)\left\|\phi_\lambda-\phi\right\|_{L_{\infty}({\Sigma^{-}_{\lambda}})},
\end{equation}
where $\left(\Sigma_\lambda^-\right)^*$ is the reflection of $\Sigma_\lambda^-$ about the plane $x=\lambda$. Note that the right-hand side of \eqref{lqnorm} is bounded in view of Corollary \ref{LPproperty of kernel}.
Since $\phi$ tends to zero at infinity, there exists $N\in \N$ such that $\|\phi\|_{L_{\infty}{\left({\left(\Sigma_\lambda^-\right)^*}\right)}} < \frac{c}{2(\|H_c\|_{L_{1}({\R})} +1)}$ for all $\lambda \leq -N$.
Then, \eqref{lqnorm} implies that $\|\phi-\phi_\lambda\|_{L_\infty({\Sigma^{-}_{\lambda}})}=0$  for $\lambda\leq-N$. As a consequence $\Sigma^{-}_{\lambda}$ must be of measure zero. Since  $\Sigma^{-}_{\lambda}$ is open, we deduce that $\Sigma^{-}_{\lambda}$ is empty for $\lambda\leq-N$.

\begin{remark}\emph{
Relation \eqref{lqnorm} turns out to be crucial when applying the method of moving planes. Note that it states in particular that if $\|\phi\|_{L_{\infty}\left({\left(\Sigma_\lambda^-\right)^*}\right)}$ is sufficiently small (which depends either on the norm of $\phi$ on the fixed set $\left(\Sigma_\lambda^-\right)^*$ or on the size of $\left(\Sigma_\lambda^-\right)^*$), then $\left\|\phi-\phi_\lambda\right\|_{L_\infty({\Sigma^{-}_{\lambda}})}=0$.
}
\end{remark}

The following theorem is the main result, which proves the symmetry of solitary waves solutions to the Whitham equation.

\begin{thm}[Symmetry of solitary-wave solutions]
\label{Symmetry_of_traveling_waves}
Let $\phi$ be a supercritical solution to the steady Whitham equation tending to zero at infinity.
Starting at a point $\lambda=-N$, where $N>0$, such that  $\Sigma_\lambda^-$ is empty for all $\lambda \leq -N$, and moving the plane $x=\lambda$ to the right as long as 
\[\phi(x)\geq \phi_\lambda(x)\qquad \mbox{for all}\quad x\in \Sigma_\lambda,\]
this process stops only and finally at some point $x=\lambda_0$, where $\phi= \phi_{\lambda_0}$ on $\Sigma_{\lambda_0}$. In particular, $\phi$ is symmetric about $\lambda_0$ and (exponentially) decreasing on the half line $[\lambda_0, \infty)$.
\end{thm}

\begin{proof}
Clearly, there can not be any crest at a point $x\leq-N$, since $\Sigma_\lambda^-$ is empty for all $\lambda<-N$.
The process of moving the plane $x=\lambda$ to the right will stop at or before it reaches a crest.
Assume that the moving plane stops at a point $x=\lambda_0 $, where $\phi(x)\geq \phi_{\lambda_0}(x)$, but $\phi(x)\neq \phi_{\lambda_0}(x)$ for all $x\in\Sigma_{\lambda_0}$, that is, $\phi$ is not symmetric about $x=\lambda_0$.
The touching lemma (Lemma \ref{touching lemma}) ensures that 
 $\phi(x)>\phi_{\lambda_0}(x)$ for all $x\in\Sigma_{\lambda_0}$ so that 
  $\overline{\Sigma_{\lambda_0}^-}$ has measure zero. By continuity of $\phi$, we have that for any $\e>0$ there exists $\delta>0$ such that $|\overline{\Sigma_\lambda^-}|<\e$ for all $\lambda\in [\lambda_0,\lambda_0+\delta)$.
It follows that there exists $\delta>0$ such that
\[\frac{1}{c}\|\phi\|_{L_{\infty}\left({\left(\Sigma_\lambda^-\right)^*}\right)}(\|H_c\|_{L_{1}({\R})}+1)<\frac{1}{4}\]
 for all $\lambda\in[\lambda_0,\lambda_0+\delta)$.
From \eqref{lqnorm} we deduce that $\|\phi-\phi_{\lambda}\|_{L_\infty({\Sigma^{-}_{\lambda}})}=0$. Therefore, $\Sigma_\lambda^-$ must be empty for all $\lambda \in[\lambda_0,\lambda_0+\delta)$ and the plane $x=\lambda_0$ can be moved further to the right, which is a contradiction. The assertion about exponential decay follows from Proposition \ref{exponential decay}.
\end{proof}

\begin{remark}\emph{In order to complete the picture of supercritical solitary solutions, we refer to a result in \cite{EW}, where the authors prove that any nonconstant even solution $\phi\in BUC^1(\R)$ of the steady Whitham equation which is nonincreasing on a half-line $(\lambda_0,\infty)$ satisfies 
\[\phi^\prime(x)<0\qquad\mbox{and}\qquad \phi(x)<\frac{c}{2}\qquad \mbox{for all}\quad x\in(\lambda_0,\infty).\]
}
\end{remark}

\bigskip

\section{Steadiness of symmetric waves}\label{ST}

We say a  solution $u$  is \emph{symmetric}, if there exists a function $\lambda\in C^1(\R_+)$ such that for every $t\geq0$ and $x\in\R$
\[u(t,x)= u(t,2\lambda(t)-x).\]
Then, $\lambda$ is called the axis of symmetry. In \cite{EHR} a local principle is established relating the property of a priori symmetry to steadiness. In particular it is proved that for a large class of local partial differential equations any classical, symmetric, unique solution constitutes a traveling wave. We follow the idea of the local principle and validate the analog result for classical, symmetric solutions to the nonlocal Whitham equation
\begin{equation}\label{WE}
u_t +2uu_x +K*u_x=0 .
\end{equation}
Recall that the kernel $K$ is given by 
\[ K(\xi)=\mathcal{F}^{-1} \bigg(\left(\frac{\tanh(\xi)}{\xi}\right)^{\frac{1}{2}}\bigg).\]
 
\begin{thm}[Symmetric solutions are traveling] 
Any classical, symmetric solution on $\R$ of the Whitham equation \eqref{WE}, which is unique with respect to initial data, is a traveling-wave solution. 
\end{thm}

\begin{proof}
Let $u$ be a classical, symmetric solution to the Whitham equation, that is $u(t,x)= u(t,2\lambda(t)-x)$ for all $(t,x)\in [0,\infty)\times \R$. The symmetry property implies that
\begin{align*}
u_t (t,x)&= u_t(t,2\lambda(t)-x)+2\dot{\lambda}(t)u_x(t,2\lambda(t)-x),\\[5pt]
u_x(t,x)&=-u_x(t,2\lambda(t)-x).
\end{align*}
Thanks to the symmetry of $K$, an easy computation shows that
\[K*u_x(t,x)=-K*u_x(t,2\lambda(t)-x).\]
Hence, the solution $u$ satisfies
\[u_t(t,2\lambda(t)-x)+2\dot{\lambda}(t)u_x(t,2\lambda(t)-x)-2uu_x(t,2\lambda(t)-x) -K*u(t,2\lambda(t)-x)=0.\]
Since the above equality is valid for all $(t,x)\in [0,\infty)\times \R$, we deduce in particular that
\begin{equation}\label{Wa}
u_t(t,x)+2\dot{\lambda}(t)u_x(t,x)-2uu_x(t,x)- K*u_x(t,x)=0.
\end{equation}
Subtracting \eqref{Wa} from \eqref{WE} yields
\begin{equation}\label{Sub}
u_x(t,x)(2u(t,x)-\dot{\lambda}(t))+K*u_x(t,x)=0\qquad \mbox{for all}\quad (t,x)\in [0,\infty)\times \R.
\end{equation}
From here, we follow the lines in \cite{EHR}. Fix a time $t_0\geq0$ and set $c:= \dot{\lambda}(t_0)$. Defining the traveling wave
\[\tilde{u}(t,x):= u(t_0, x-c(t-t_0)),\]
one obtains that $\tilde{u}$ satisfies
\begin{align*}
&\tilde{u}_t(t,x)+2\tilde{u}\tilde{u}_x(t,x)+K*\tilde{u}(t,x) \\
&\;= -c{u}_x(t_0, x-c(t-t_0)) +2{u}{u}_x(t_0, x-c(t-t_0))+K*{u}(t_0, x-c(t-t_0))\\[5pt]
&\;= {u}_x(2{u}-c)(t_0, x-c(t-t_0))+K*u(t_0, x-c(t-t_0))=0,
\end{align*}
due to \eqref{Sub}. Hence, $\tilde{u}$ is a solution to the Whitham equation.
By construction it holds $\tilde{u}(t_0,x) = u(t_0,x)$
so that $\tilde{u}$ coincides with the symmetric solution $u$ at $t=t_0$. By uniqueness of solutions with respect to initial data, we conclude that $u$ is indeed a traveling-wave solution.
\end{proof}

\subsection*{Acknowledgments}
The authors express their gratitude to Jean--Claude Saut for proposing this topic and to Vera Mikyoung Hur for fruitful discussions, which helped to remove additional assumptions on the wave profile when applying the method of moving planes.
This work was supported by the projects "Nonlinear Water Waves" (Grant No. 231668) and  "Waves and Nonlinear Phenomena" (Grant No. 250070) of the Research Council of Norway. 

\appendix
\section{}

We collect some technical lemmata, which are used to prove the decay result for supercritical solitary-wave solutions of the Whitham equation.

\begin{lem}\label{main inequality lemma}
	Let $l$ and $m$ be two constants satisfying $0<l<m-1$. Then, there exists $B=B(l,m)>0$, such that for all $\e\in (0,1)$ the following inequality holds true:
	\[ \int_{\R}\frac{|y|^{l}}{(1+\e|y|)^{m}(1+|x-y|)^{m}}\dd y\leq \frac{B|x|^{l}}{(1+\e|x|)^{m}} \quad \mathrm{for ~all} ~ x\in \R, ~ |x|\geq 1,\]
\end{lem}
A proof of Lemma \ref{main inequality lemma} can be found in \cite[Lemma 3.1.1]{bona1997decay}.
The following two lemmata are results, which can be proved easily by induction and are applied in the proof of Theorem \ref{exponential decay 1}.
\begin{lem}\label{Alemma} Let $n\in \N$ and $q\geq 1$, then
\begin{equation*}
(qn)! \leq [q^{n}(n!)]^q.
\end{equation*}
\end{lem}

\begin{lem}\label{convolution lemma}
Let $f$ and $g$ be functions belonging to $L_1(\R)$, such that there exists $N\in \N$ with  $(\cdot)^nf$ and $(\cdot)^ng$ are bounded in $L_1(\R)$ for all $n\leq N$. Then
\begin{equation*}
(x)^n(f*g)(x)= \sum_{j=0}^n \binom{n}{j} ((\cdot)^{n-j}f*(\cdot)^jg)(x).
\end{equation*}
\end{lem}

\bibliographystyle{siam}
\bibliography{Bib_Whitham}

\end{document}